\documentclass[11pt,reqno]{amsart}

\usepackage{amsmath, amsfonts, amsthm, amssymb, graphicx, cite}

\usepackage{amssymb}
\usepackage{amsmath}
\usepackage{color}
\usepackage{graphics}
\usepackage{epsfig}
\usepackage{hyperref}
\usepackage{epstopdf}

\allowdisplaybreaks

\newtheorem{claim}{\bf \t}[part]

\theoremstyle{plain}
\newtheorem{Theorem}{Theorem}[section]
\newtheorem{Lemma}{Lemma}[section]

\theoremstyle{definition}
\newtheorem{Definition}{Definition}[section]

\theoremstyle{remark}
\newtheorem{Remark}{Remark}[section]

\theoremstyle{definition}

\numberwithin{equation}{section}

\def\R{\mathbb{R}}

\begin{document}
\title[Isometric Immersion of Surface and Lax-Friedrichs scheme]
{Isometric Immersion of Surface with Negative Gauss Curvature and the Lax-Friedrichs Scheme}
\author{Wentao Cao \and Feimin Huang \and Dehua Wang}

\address{W. Cao, Institute of Applied Mathematics, AMSS, CAS, Beijing 100190, China.}
\email{cwt@amss.ac.cn}

\address{F. Huang, Institute of Applied Mathematics, AMSS, CAS, Beijing 100190, China.}
 \email{\tt fhuang@amt.ac.cn}

\address{D. Wang, Department of Mathematics, University of Pittsburgh,
                Pittsburgh, PA 15260, USA.}
\email{\tt dwang@math.pitt.edu}


\date{\today}
\thanks{}

\begin{abstract}
The isometric immersion of two-dimensional Riemannian manifold with negative Gauss curvature into the three-dimensional Euclidean space is considered through  the Gauss-Codazzi equations for the first and second fundamental forms.
The large $L^\infty$ solution is obtained which leads to a $C^{1,1}$ isometric immersion.  The approximate solutions are constructed by the Lax-Friedrichs finite-difference scheme with the fractional step. The uniform estimate is established by studying the equations satisfied by the Riemann invariants and using the sign of the nonlinear part. The $H^{-1}$ compactness is also derived. A compensated compactness framework is applied to obtain the existence of large $L^\infty$ solution to the Gauss-Codazzi equations for the surfaces more general  than those in literature.
\end{abstract}

\maketitle

\section{Introduction}

The isometric embedding or immersion of two-dimensional Riemannian manifolds  into the three-dimensional Euclidean space is a classical problem in geometry. It can be formulated as an elliptic problem if the Gauss curvature is positive, a hyperbolic problem if the curvature is negative, and a mixed elliptic-hyperbolic type problem if the curvature changes signs. This problem has been extensively studied in literature and we refer the readers to the book \cite{HH}, the papers \cite{Han,H,Lin} and the references therein.
In particular, there have been many studies on the elliptic case of the isometric embedding problem with positive curvature; see the book \cite{HH}. For the case of mixed type when the curvature changes signs, Lin in \cite{Lin} and Han in \cite{Han} obtained the local isometric embedding of surfaces when the  Gauss curvature changes signs cleanly using the approach of the symmetric positive differential system.
For the hyperbolic case with negative curvature,  Hong in \cite{H} proved  the smooth  isometric embedding  of surfaces when the  Gauss curvature decays at  certain rate in the time-like direction and  the  $C^1$ norm of  the initial data  is small; the isometric immersion with large data was obtained in Chen-Slemrod-Wang \cite{CSW} by a fluid dynamic formulation and a vanishing viscosity method for catenoid type surfaces, and in Cao-Huang-Wang  \cite{CHW} by the artificial viscosity method for both the catenoid and helicoid type surfaces; and the isometric immersion with small BV data was also considered in Christoforou \cite{Christoforou}.
The purpose of this paper is to study the isometric immersion with large data for more general surfaces with negative curvatures.
We remark that the results in \cite{CSW,CHW} and this paper hold for decay of order $t^{-4}$  in the Gauss curvature, and
the method of compensated compactness was recently applied in Christoforou-Slemrod \cite{ChSlemrod} to the isometric immersion of surfaces with rough data and slowly decaying negative Gauss curvature in the order $t^{-2-\delta/2}, \, \delta\in(0,4)$,  the same decay rate as in \cite{H}.

The classical surface theory indicates that  the isometric embedding or immersion can be realized if the first fundamental form and the second fundamental form of surfaces satisfy the Gauss-Codazzi equations (cf. \cite{BS,M1,M2,PS}).
When the curvature is negative, the Gauss-Codazzi system is a $2\times2$ hyperbolic system of balance laws. Our goal of this paper is to obtain the large $L^\infty$ solution of the Gauss-Codazzi system so that the $C^{1,1}$ isometric immersion can be established (\cite{M1,M2}). As mentioned above, the large $L^\infty$ solution of the Gauss-Codazzi system was obtained in \cite{CHW, CSW} for certain type surfaces with special given metrics. To obtain the large $L^\infty$ solution for more general given metrics of surfaces,
we shall reformulate the system as an inhomogeneous system of Chaplygin gas dynamics, use the Lax-Friedrichs type finite-difference scheme with the fractional step (cf. \cite{DCL, HLY})  to construct the approximate solutions, and apply the compensated compactness framework in \cite{CSW} to show the convergence and obtain the global large $L^\infty$ solution of the system.
As in \cite{CHW,CSW}, the hardest and most important part is  to establish the  $L^\infty$ uniform estimate. In this paper  we shall take the advantage of the sign of nonlinear part in the discrete ordinary differential equations (ODEs) satisfied by the Riemann invariants to obtain the uniform $L^\infty$ estimate via the invariant regions. One of our observations is that
the invariant regions are contained only in the region where the two Riemann invariants have different signs, which is also true in
\cite{CHW,CSW}. Through the fractional Lax-Friedrichs scheme, we obtain the uniform $L^\infty$ bound of the approximate solutions with arbitrarily large initial data. Our results contain more manifolds than those in \cite{CHW}, in which only helicoid and catenoid type manifolds are included.
Then we apply the compensated compactness framework proved in \cite{CSW} for the Gauss-Codazzi system to obtain the existence of global weak solution in $L^\infty$ which yields an isometric immersion of surfaces.

The rest of the paper is organized as follows. In Section 2, we first normalize the state variables and reformulate the Gauss-Codazzi system as an  inhomogeneous system of Chaplygin gas dynamics. Then we  recall some basic properties of the Riemann solutions to the homogeneous Chaplygin gas dynamics, construct the approximate solutions to the inhomogeneous system with the help of Lax-Friedrichs scheme and the fractional step, and give  the definition of weak solutions. In Section 3, we establish the  $L^\infty$ estimate for more general surfaces  by  studying carefully the equations satisfied by the Riemann invariants. In Section 4, we prove the  $H^{-1}_{loc}$ compactness using the method of \cite{DCL}. In Section 5, combining the above two estimates and the compensated compactness framework in \cite{CSW}  we state and prove our main theorem on the existence of weak solution to the Gauss-Codazzi system.

\bigskip

\section{Reformulation and Approximate Solutions}

For a two-dimensional Riemannian manifold  parametrized by $(x_1,x_2)=(x,y)$ in an open set of $\R^2$ with the given metric $g_{ij}, i, j=1, 2$, and the second fundamental form $h_{11}(dx)^2+ 2h_{12}dxdy + h_{22}(dy)^2$,
as in \cite{CSW, HH} the Gauss-Codazzi system for the isometric embedding of the two-dimensional Riemannian manifold into $\mathbb{R}^3$ is the following:
\begin{eqnarray}\label{GC}
\left\{
\begin{array}{lll}
 M_x-L_y=\Gamma^2_{22}L-\Gamma^2_{12}M+\Gamma^2_{11}N,\\
 N_x-L_y=-\Gamma^1_{22}L+\Gamma^1_{12}M-\Gamma^1_{11}N,\\
 LN-M^2=K,
\end{array}
\right.
\end{eqnarray}
where
$$
L=\frac{h_{11}}{\sqrt{|g|}}, \qquad M=\frac{h_{12}}{\sqrt{|g|}},
\qquad N=\frac{h_{22}}{\sqrt{|g|}},
$$
$|g|=det(g_{ij})=g_{11}g_{22}-g_{12}^2$, $K$ is the Gauss
curvature which is equal to the following:
$$
K=\frac{g_{2l}}{|g|}\left(\partial_1\Gamma^{l}_{12}-\partial_2\Gamma^{l}_{11}
+\Gamma^{i}_{12}\Gamma^{l}_{i1}-\Gamma^{i}_{11}\Gamma^{l}_{i2}\right),
$$
 and
$$
\Gamma_{ij}^{l}=\frac12g^{lr}\left(\partial_j g_{ir}+\partial_i g_{jr}-\partial_r g_{ij}\right)
$$
is the Christoffel symbol with $(g^{lr})$
denoting the inverse of $(g_{ij})$ and
$(\partial_1,\partial_2)=(\partial_x, \partial_y)$.
Here the Einstein summation convention is used.

In this paper, we consider the isometric embedding
of a complete simply connected two-dimensional manifold with negative Gauss curvature
$$K=-k^2$$ for some positive smooth function $k$. There is a global geodesic coordinate system (\cite{HH})
and the metric of the manifold  is $$ds^2=dy^2+B(x,y)^2dx^2,$$
where $B(x,y)>0$ satisfies
\begin{eqnarray}\label{B}
\left\{
\begin{array}{lllll}
B_{yy}=k^2B,\\
B(x,0)=1, \quad B_y(x,0)=0.
\end{array}
\right.
\end{eqnarray}
Therefore $B_y(x,y)\leq B_y(x,0)=0$ for $y\leq 0,$
and $|g|=B^2.$

For any given $T>0$,  we set $y=t-T$ for $t\in[0, T],$ and
$$\tilde{L}=\frac{L}{k},\quad
\tilde{M}=\frac{M}{k},\quad \tilde{N}=\frac{N}{k}.$$
Then the Gauss-Codazzi system \eqref{GC} becomes
\begin{equation}\label{GC2}
\left\{
\begin{array}{lll}
 \tilde{M}_x-\tilde{L}_t=\tilde{\Gamma}^2_{22}\tilde{L}-
 2\tilde{\Gamma}^2_{12}\tilde{M}+\tilde{\Gamma}^2_{11}\tilde{N},\\
 \tilde{N}_x-\tilde{M}_t=-\tilde{\Gamma}^1_{22}\tilde{L}+
 2\tilde{\Gamma}^1_{12}\tilde{M}-\tilde{\Gamma}^1_{11}\tilde{N},\\
 \tilde{L}\tilde{N}-\tilde{M}^2=-1,
\end{array}
\right.
\end{equation}
where
\begin{eqnarray*}
&&\tilde{\Gamma}^1_{11}=\Gamma^1_{11}+\frac{k_x}{k}= \frac{B_x}{B}+\frac{k_x}{k}, \quad
\tilde{\Gamma}^2_{11}=\Gamma^2_{11}=-BB_t,\\
&&\tilde{\Gamma}^1_{12}=\Gamma^1_{12}+\frac{k_t}{2k}=\frac{B_t}{B}+\frac{k_t}{2k}, \quad
\tilde{\Gamma}^2_{12}=\Gamma^2_{12}+\frac{k_x}{2k}=\frac{k_x}{2k},\\
&&\tilde{\Gamma}^1_{22}=\Gamma^1_{22}=0, \quad
\tilde{\Gamma}^2_{22}=\Gamma^2_{22}+\frac{k_t}{k}=\frac{k_t}{k}.
\end{eqnarray*}

Furthermore,  we let $\tilde{L}=\rho$, $\tilde{M}=-m$, then $\tilde{N}=\frac{m^2-1}{\rho}$.
The above system \eqref{GC2} can be rewritten as the following inhomogeneous system of
Chaplygin gas:
\begin{eqnarray}\label{Chap}
\left\{ \begin{array}{llll}
\displaystyle \rho_t+m_x=R(\rho,m,x,t):=-\rho\frac{k_t}{k}-2m\frac{k_x}{2k}-\frac{m^2-1}{\rho}(-BB_t),\\
\displaystyle  m_t+\left(\frac{m^2-1}{\rho}\right)_x=S(\rho,m,x,t):=-2m\left(\frac{B_t}{B}+\frac{k_t}{2k}\right)-
\frac{m^2-1}{\rho}\left(\frac{B_x}{B}+\frac{k_x}{k}\right),
\end{array}
\right.
\end{eqnarray}
or
\begin{eqnarray}\label{Chap2}
\left\{\begin{array}{ll}
U_t+f(U)_x=H(U,x,t),\\
U=(\rho,m)^\top, \quad f(U)=(m,(m^2-1)/\rho)^\top,\\
H(U,x,t)=(R(\rho,m,x,t),S(\rho,m,x,t))^\top.
\end{array}
\right.
\end{eqnarray}
We remark that  the $L^\infty$ weak solution of the Gauss-Codazzi system \eqref{GC} is equivalent to
the $L^\infty$ weak solution of inhomogeneous system of
Chaplygin gas \eqref{Chap}.  

\subsection{Riemann solutions of the homogenous system}
The following Cauchy problem of the homogenous system associated with
(\ref{Chap}):
\begin{eqnarray}\label{HO}
\left\{\begin{array}{ll}
{\rho}_t+m_x=0,\\
m_t+(\frac{m^2-1}{\rho})_x=0,\\
(\rho(x,0),m(x,0))=(\rho_0(x),m_0(x)),
\end{array}\right.
\end{eqnarray}
was studied in \cite{P} through the Euler-Lagrange transform, where
 the entropy solution  was obtained and it was shown that the solution has the same regularity as that of  the
 initial data.
 The Remainian invariants $w, z$
 and the eigenvalues $\lambda_i,i=1,2$
 of (\ref{HO}) are
 $$\lambda_1=z=\frac{m-1}{\rho}, \quad \lambda_2=w=\frac{m+1}{\rho},$$
and  we have
 \begin{eqnarray}\label{HOR}
\left\{\begin{array}{ll}
w_t+zw_x=0,\\
z_t+wz_x=0,\\
(w(x,0),z(x,0))=(w_0(x),z_0(x))=\left(\frac{m_0(x)+1}{\rho_0(x)},\frac{m_0(x)-1}{\rho_0(x)}\right).
\end{array}
\right.
\end{eqnarray}
If $0<\rho_0(x)\leq C$ and  $|m_0(x)|\leq C$ for some constant $C>0$, we have $w_0, z_0\in L^\infty(\mathbb{R})$ and
$\inf_xw_0(x)>\sup_xz_0(x)$, then the entropy solution in $L^\infty$ of (\ref{HO}) is equivalent to
the entropy solution in $L^\infty$ of the following problem:
\begin{eqnarray}\label{HORL}
\left\{\begin{array}{ll}
\tilde{w}_s-\tilde{w}_\xi=0,\\
\tilde{z}_s+\tilde{z}_\xi=0,\\
(\tilde{w}(s,\xi),\tilde{z}(s,\xi))=(w(t,x),z(t,x)),
\end{array}
\right.
\end{eqnarray}
where $$d\xi=\frac{2}{w-z}dx-\frac{w+z}{w-z}dt, \quad s=t.$$
 Let
\begin{equation}
Y_0(x)=\int_0^x\frac{2}{w_0(y)-z_0(y)}dy, \quad X_0(\xi)=Y_0^{-1}(\xi).
\end{equation}
Then $$\tilde{w}(s,\xi)=w_0(X_0(\xi+s)), \quad \tilde{z}(s,\xi)=z_0(X_0(\xi-s)),$$
and $$x=X(t,y), \quad dx=\frac{w-z}{2}d\xi+\frac{w+z}{2}dt$$ satisfying
$X(0,\xi)=X_0(\xi)$.
In addition, $$X(t,\xi)=\frac{1}{2}\int_0^{\xi+t}w_0(X_0(y))dy-
\frac{1}{2}\int_0^{\xi-t}z_0(X_0(y))dy.$$
For fixed $t$, define $Y(t,\cdot)=X^{-1}(t,\cdot)$. Then
$$w(t,x)=w_0(X_0(Y(t,x)+t)), \quad z(t,x)=z_0(X_0(Y(t,x)-t)).$$
Hence, if $w_0, z_0\in L^\infty(\mathbb{R})$, then
$w, z\in L^\infty(\mathbb{R}^+\times\mathbb{R})$, and moreover,
$w_0(x)$ and $w(x,t)$ have the same sign, and  $z_0(x)$ and $z(t,x)$ have the same sign.
Hence the Riemann solutions of (\ref{HO}) with piecewise constant initial data
are in the invariant region of the initial data.

More precisely, for the Riemann solution of (\ref{HO}) with the given initial data $(\rho_0, m_0)$,
the  state $(\rho, m)$   can only be connected to $(\rho_0, m_0)$ on the right
 by a contact discontinuity wave (CDW) since the hyperbolic system (\ref{HO}) is
linearly degenerate. The possible contact discontinuity waves are
\begin{eqnarray}
CDW_1:\qquad \frac{m-1}{\rho}=\frac{m_0-1}{\rho_0};\\
CDW_2:\qquad \frac{m+1}{\rho}=\frac{m_0+1}{\rho_0}.
\end{eqnarray}
Therefore, we have the following property on the Riemann solution:
\begin{Lemma}\label{lemma1}
The region $D=\{(\rho, m): 0\leq w\leq w_0, \; 0\geq z\geq z_0\}$ is the invariant
region of the Riemann problem, i.e. if $(\rho_0, m_0)\in D,$ the solutions of
the Riemann problem belong to $D$ too.
Moreover,  if $\{(\rho,m)(x):a<x<b\}\subset D,$
then
$$\left(\frac{1}{b-a}\int_a^b\rho(x)dx,\frac{1}{b-a}\int_a^bm(x)dx\right)\in D.$$
\end{Lemma}
The lemma can be proved by combining Figure \ref{figure-contact} and \cite{P}, since the invariant
region is convex  and thus Jensen's inequality can be applied. We omit the details.

\begin{figure}
\begin{center}
  \includegraphics[width=12cm,height=8cm]{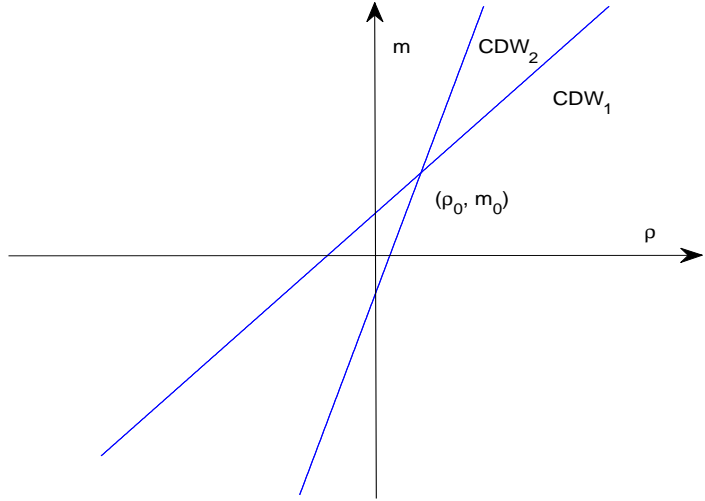}\\
  \caption{Graph of contact discontinuity waves}\label{figure-contact}
  \end{center}
\end{figure}

\subsection{Construction of the  approximate solutions}

We now use the Lax-Friedrichs scheme with the fractional step to construct the approximate solutions
of the initial-value problem of  (\ref{Chap}) (or \eqref{Chap2}) for $x\in\mathbb{R}, t\in [0,T]$ with any given $T>0$.
This is equivalent to solving (\ref{GC}) for $y\in[-T,0], x\in\mathbb{R}$.

Let $h>0, l>0$ be the time and space meshes respectively, satisfying the CFL  condition:
$$\max_{i=1,2}\lambda_i(\rho^l,m^l)<l/2h\leq \Lambda$$
for some constant $\Lambda>0$. Denote
\begin{eqnarray*}
&&S_n=\{(x,t): nh\leq t<(n+1)h\},\\
&&I_n=\{j: n+j \text{ is even}\}, \\
&&Q_{nj}=\{(x,t): (j-1)l<x<x(j+1)l, \; nh\leq t<(n+1)h,  \; j+n \text{ is even}\}.
\end{eqnarray*}
for any integers $j,n\geq0$.

We construct  the approximate solutions $U^l=(\rho^l,m^l)$  of the  initial-value problem of \eqref{Chap2} with initial data $U(x, 0)=U_0(x)$ as follows.

\emph{Firstly}, for $ (x,t)\in Q_{0j}$, we define
$$U^l(x,t)=U^l_R(x,t)+H(U^l_R(x,t),x,t)t,$$
where $U^l_R(x,t)=(\rho^l_R(x,t),m^l_R(x,t))$ is the Riemanian solution of (\ref{HO})
with the initial data
\begin{eqnarray}
U^l_R(x,0)=\left\{\begin{array}{ll}
U^l_R((j-1)l,0), \; x<jl;\\
U^l_R((j+1)l,0), \; x>jl.
\end{array}
\right.
\end{eqnarray}
Here $U^l_R(x,0)=U_0(x)\chi_l(x)$, and
$\chi_l(x)$ is the characteristic function of
$[-1/l,1/l].$
Define
$$ U^1_j=\frac{1}{2l}\int_{(j-1)l}^{(j+1)l}U^l(x,h-0)dx.$$
\par
\emph{Secondly}, suppose $U^l(x,t)$ has been defined for $t<nh$.
Take
$$U^{n}_j=\frac{1}{2l}\int_{(j-1)l}^{(j+1)l}U^l(x,nh-0)dx.$$
Then
for $(x,t)\in Q_{nj}$, we define
$$U^l(x,t)=U^l_R(x,t)+H(U^l_R(x,t),x,t,h)(t-nh),$$
where $U^l_R(x,t)=(\rho^l_R(x,t),m^l_R(x,t))$ is the Riemanian solution of (\ref{HO})
with the initial data
\begin{eqnarray}
U^l_R(x,nh)=\left\{\begin{array}{ll}
U^{n}_{j-1}, \; x<jl,\\
U^{n}_{j+1}, \; x>jl;
\end{array}\right.
\end{eqnarray}
and take
$$U^{n+1}_j=\frac{1}{2l}\int_{(j-1)l}^{(j+1)l}U^l(x,(n+1)h-0)dx.$$


In the rest of the paper, we shall show that the approximate solutions $U^l(x,t)$ constructed above have a subsequence converging to the weak solution $U(x,t)$ of the initial-value problem (\ref{Chap2}) with initial data $U_0(x)=(\rho_0(x), m_0(x))^\top$ for
$(x,t)\in\Pi_T=\mathbb{R}\times[0,T]$ in the following sense:
\begin{equation*}
\iint_{\Pi_T}\left(U(x,t)\varphi_t+f(U(x,t))\varphi_x
H(U(x,t),x,t)\varphi\right)dxdt+\int_\mathbb{R}U_0(x)\varphi(x,0)dx=0,
\end{equation*}
where $\varphi(x,t)$ is any smooth function whose support is compact in the region $\Pi_T$ and $\varphi(x, T)=0$.

\bigskip

\section{$ L^{\infty}$ Uniform Estimate}

In this section we shall establish the uniform bound of the approximate solutions using the
Riemann invariants. For the above approximate solutions, we denote the
Riemann invariants of the homogeneous system (\ref{HO}) by
 $$w^l_R=\frac{m^l_R+1}{\rho^l_R},\quad z^l_R=\frac{m^l_R-1}{\rho^l_R},$$
 and the Riemann invariants of  the original system (\ref{Chap}) by
$$w^l=\frac{m^l+1}{\rho^l},\quad z^l=\frac{m^l-1}{\rho^l}.$$
 Then for $nh\leq t<(n+1)h, n\geq0,$ one has
 \begin{eqnarray*}
w^l(x,t)&=&\frac{m^l_R+1+\left[-2m^l_R\big(\frac{B_t}{B}+\frac{k_t}{2k}\big)-
\frac{(m^l_R)^2-1}{\rho^l_R}\big(\frac{B_x}{B}+\frac{k_x}{k}\big)\right](t-nh)}
{\rho^l_R+\left[-\rho^l_R\frac{k_t}{k}-2m^l_R\frac{k_x}{2k}-\frac{(m^l_R)^2-1}{\rho^l_R}(-BB_t)\right](t-nh)}\\
&=&\frac{w^l_R-\left[(w^l_R+z^l_R)\big(\frac{B_t}{B}+\frac{k_t}{2k}\big)+
w^l_Rz^l_R\big(\frac{B_x}{B}+\frac{k_x}{k}\big)\right](t-nh)}
{1+\left[-\frac{k_t}{k}-(w^l_R+z^l_R)\frac{k_x}{2k}-w^l_Rz^l_R(-BB_t)\right](t-nh)},
\end{eqnarray*}
 \begin{eqnarray*}
z^l(x,t)&=&\frac{m^l_R-1+\left[-2m^l_R\big(\frac{B_t}{B}+\frac{k_t}{2k}\big)-
\frac{(m^l_R)^2-1}{\rho^l_R}\big(\frac{B_x}{B}+\frac{k_x}{k}\big)\right](t-nh)}
{\rho^l_R+\left[-\rho^l_R\frac{k_t}{k}-2m^l_R\frac{k_x}{2k}-\frac{(m^l_R)^2-1}{\rho^l_R}(-BB_t)\right](t-nh)}\\
&=&\frac{z^l_R-\left[(w^l_R+z^l_R)\big(\frac{B_t}{B}+\frac{k_t}{2k}\big)+
w^l_Rz^l_R\big(\frac{B_x}{B}+\frac{k_x}{k}\big)\right](t-nh)}
{1+\left[-\frac{k_t}{k}-(w^l_R+z^l_R)\frac{k_x}{2k}-w^l_Rz^l_R(-BB_t)\right](t-nh)}.\\
\end{eqnarray*}
Thus,
\begin{eqnarray*}
&&w^l=w^l_R+F(w^l_R,z^l_R,x,t,h)(t-nh),\\
&&z^l=z^l_R+F(z^l_R,w^l_R,x,t,h)(t-nh),
\end{eqnarray*}
where
\begin{eqnarray*}
F(w^l_R,z^l_R,x,t,h)&=&\frac{-(\frac{B_t}{B}-\frac{k_t}{2k})w^l_R-(\frac{B_t}{B}+\frac{k_t}{2k})z^l_R
-BB_t(w^l_R)^2(z^l_R)}{1-\left[\frac{k_t}{k}+\frac{k_x}{2k}(w^l_R+z^l_R)-BB_tw^l_Rz^l_R\right](t-nh)}\\
&&+\frac{-(\frac{B_x}{B}+\frac{k_x}{2k})w^l_Rz^l_R+\frac{k_x}{2k}(w^l_R)^2}
{1-\left[\frac{k_t}{k}+\frac{k_x}{2k}(w^l_R+z^l_R)-BB_tw^l_Rz^l_R\right](t-nh)}.
\end{eqnarray*}
It is challenging  to establish the uniform bound of $w^l(x,t), z^l(x,t)$ for general $B(x,t)$ and $k(x,t)$.
Assume that the given metric depends only on $t,$ and we still denote $k(x,t)=k(t), B(x, t)=B(t).$ We note $B'(t)\leq0$ as mentioned before. Then we have the following uniform estimate.

\begin{Theorem}\label{Th1}
Suppose that $k$ and $B$ depend only on $t$,
$\ln (B^2k)$ is nondecreasing  and  smooth,
and the initial data $(\rho_0(x), m_0(x))$ satisfies the following condition:
\begin{equation*}
\delta_0\leq w(x, 0)\leq P_0,~-P_0\leq z(x, 0)\leq0 \; (\text{or }
0\leq w(x, 0)\leq P_0,~-P_0\leq z(x, 0)\leq-\delta_0),
\end{equation*}
for some constants $\delta_0>0, P_0>0$ and all $x\in\mathbb{R}$.
Then there exists a constant $h_0$ such that,  for $h\leq h_0$, the  approximate solutions
are uniformly bounded in the region $\Pi_T=\mathbb{R}\times[0,T]$, i.e.,
there exist  constants $P(T),A(T)>0$ to be determinated later, such that
\begin{equation*}
\begin{split}
&\delta_0e^{-A(T)T}\leq w^l(x, t)\leq P(T),~-P(T)\leq z^l(x, t)\leq0\\
(\text{or }
&0\leq w^l(x, t)\leq P(T),~-P(T)\leq z^l(x, t)\leq-\delta_0e^{-A(T)T}).
\end{split}
\end{equation*}
\end{Theorem}
\begin{Remark}
The condition that $\ln (B^2k)$ is nondecreasing plays essential role in Theorem \ref{Th1}.
$\frac{d}{dt}\ln (B^2k)\geq0$ is equivalent to, in the original variable $y$,
$$\frac{2B'}{B}+\frac{k'}{k}\geq0, \text{ when } y\leq0,$$
where $'$ denotes $\frac{d}{dy}.$ If we take
\begin{equation} \label{KB}
\frac{k'}{k}=a\frac{B'}{B},
\end{equation}
for some function $a=a(y)$, then
 $$\frac{2B'}{B}+\frac{k'}{k}=\frac{B'}{B}(a+2)\geq0,$$
if $a\leq-2$ due to the fact that $B'\leq0.$

When the function $a$ in \eqref{KB}  is constant, the corresponding manifold is a helicoid type surface, and this special case has been investigated in \cite{CHW} through the vanishing viscosity method and maximum principle.
However,  the function $a$ in \eqref{KB} is not constant in general, that is $a=a(y).$ Thus more manifolds are included. For example, take
\begin{equation*}
\begin{split}
&B(y)''=\frac{1}{(1+y^2)^2}, \quad B(y)'= \int_0^y\frac{ds}{(1+s^2)^2},\\ &B(y)=1+\int_0^y\int_0^\tau\frac{ds}{(1+s^2)^2}d\tau,\\
&k(y)=\sqrt{\frac{B''(y)}{B(y)}}=\frac{1}{\sqrt{(1+y^2)^2(1+\int_0^y\int_0^\tau\frac{ds}{(1+s^2)^2}d\tau)}}.
\end{split}
\end{equation*}
A simple calculation yields
\begin{equation*}
\begin{split}
a(y)&=\frac{k(y)'}{k(y)}\frac{B(y)}{B(y)'}=
\frac{\left(\frac{B(y)''}{B(y)}\right)'}{2\frac{B(y)''}{B(y)}}\frac{B(y)}{B(y)'}
=\frac{B(y)'''B(y)-B(y)''B(y)'}{2B(y)''B(y)'}\\
&=\frac{-4y-4y\int_0^y\int_0^\tau\frac{ds}{(1+s^2)^2}d\tau-(1+y^2)\int_0^y\frac{ds}{(1+s^2)^2}}
{2(1+y^2)\int_0^y\frac{ds}{(1+s^2)^2}}.
\end{split}
\end{equation*}
It is obvious that $a=a(y)$ is not constant (in fact $a'(y)\not\equiv 0$). Moreover,
$$a(y)+2=\frac{-4y-4y\int_0^y\int_0^\tau\frac{ds}{(1+s^2)^2}d\tau+3(1+y^2)\int_0^y\frac{ds}{(1+s^2)^2}}
{2(1+y^2)\int_0^y\frac{ds}{(1+s^2)^2}}$$
whose denominator is negative since $y<0,$ and numerator is
$$f(y)=-4y-4y\int_0^y\int_0^\tau\frac{ds}{(1+s^2)^2}d\tau+3(1+y^2)\int_0^y\frac{ds}{(1+s^2)^2}.$$
By a direct computation, we have
\begin{equation*}
\begin{split}
&f(y)'=-4-4\int_0^y\int_0^\tau\frac{ds}{(1+s^2)^2}d\tau+2y\int_0^y\frac{ds}{(1+s^2)^2}+\frac{3}{1+y^2},\\
&f(y)''=-2\int_0^y\frac{ds}{(1+s^2)^2}-\frac{4y}{(1+y^2)^2}\ge 0,
\end{split}
\end{equation*}
which implies $f(y)'\leq f(0)'=-1<0$ for $y\leq0$. Therefore $f(y)\geq f(0)=0$ and then $a(y)+2\leq0$. That is, $a(y)\le -2$ but $a(y)$ is not constant.

On the other hand, it is not easy to apply the vanishing viscosity method and maximum principle in \cite{CHW} for the case that $a(y)$ is not constant. In fact, the source term $\psi_1(\tilde{u},\tilde{v}), \psi_2(\tilde{u},\tilde{v})$ in \cite{CHW} are
\begin{equation*}
\begin{split}
&\psi_1(\tilde{u},\tilde{v})=\tilde{u}-(a(y)+1)\tilde{v}+(\tilde{u}^2-\tilde{v}^2)(\tilde{u}+\tilde{v}),\\
&\psi_2(\tilde{u},\tilde{v})=\tilde{u}+(a(y)+1)\tilde{v}+(\tilde{u}^2-\tilde{v}^2)(\tilde{u}-\tilde{v}).
\end{split}
\end{equation*}
Since $a(y)$ depends on $y,$ $\psi_1(\tilde{u},\tilde{v})=0$ and $\psi_2(\tilde{u},\tilde{v})=0$
are curves with parameter $y$ in $(\tilde{u}, {\tilde{v}})$ plane. Any square in Figure 3 or Figure 5
of \cite{CHW} shall not be the invariant region for the system. Thus, it is difficult to find an invariant region for $\tilde{u},\tilde{v}$ by the approach in \cite{CHW}.
This is why we adopt the fractional Lax-Friedrichs  method in this paper which allows the $L^\infty$ bound of the approximate solutions  to increase with  time $t$. That is,
the region for solution in the phase plane may expand with time $t$.
\end{Remark}

\begin{proof}
We  prove the theorem only for the case that
$$\delta_0\leq w(x, 0)\leq P_0,~-P_0\leq z_0(x)\leq0,$$
and the other case can be treated in the same way. The proof is divided into two steps.\par
\noindent\textit{Step 1:} We prove that there exists a constant $h_0>0$
 such that, when $h\leq h_0$, the Riemann invariants  satisfy the following:
\begin{equation}\label{Rie}
0\leq w^l(x,t)\leq P(T),\quad -P(T)\leq z^l(x,t)\leq0,
\end{equation}
for $(x,t)\in\mathbb{R}\times[0,T],$
where $$P(T)=P_0\frac{k(T)^2}{k(0)^2},$$ with $P_0=\max\{\sup_xw(x, 0),
-\inf_xz(x, 0)\}$.

Indeed, for $nh\leq t<(n+1)h$, from the assumption on $k$, one has
\begin{eqnarray*}
&&w^l=w^l_R+(t-nh)\frac{-(\frac{B'}{B}-\frac{k'}{2k})w^l_R-(\frac{B'}{B}+\frac{k'}{2k})z^l_R
-BB'(w^l_R)^2(z^l_R)}{1-\left[\frac{k'}{k}-BB'w^l_Rz^l_R\right](t-nh)},\\
&&z^l=z^l_R+(t-nh)\frac{-(\frac{B'}{B}-\frac{k'}{2k})z^l_R-(\frac{B'}{B}+\frac{k'}{2k})w^l_R
-BB'(z^l_R)^2(w^l_R)}{1-\left[\frac{k'}{k}-BB'w^l_Rz^l_R\right](t-nh)}.
\end{eqnarray*}
\par
We first assume  that  there exists a constant $h^*$ such that,  for $h\leq h^*$,
$$0\leq w^l(x,t)\leq 2P(T),\quad -2P(T)\leq z^l(x,t)\leq0.$$
Because of the above a priori assumption, there exists a constant $0<h_1=h_1(P(T))<h^*$ such that
for any $(x,t)\in\Pi_T$, when $h<h_1$,
$$h\left|\frac{k'}{k}-BB'w^l_Rz^l_R\right|\leq\frac{1}{2}.$$
And then
\begin{equation*}
D(x,t,h):=1-\left(\frac{k'}{k}-BB'w^l_Rz^l_R\right)(t-nh)\in\left[\frac{1}{2}, \frac{3}{2}\right].
\end{equation*}
\par
The estimate \eqref{Rie} holds for $t=0$ since we choose the initial data satisfying  the condition.
For $0<t<h,$ from the properties of the Riemann solution of the homogeneous system with the given initial data,
we see that
$$P(T)\geq P_0\geq w^l_R(x,t)\geq0, \quad -P(T)\leq-P_0\leq z^l_R(x,t)\leq0.$$
Since $\ln(B^2k)$ is nondecreasing  and $B$ is decreasing,
then $\frac{B'}{B}+\frac{k'}{2k}\geq0$ and  $\frac{B'}{B}-\frac{k'}{2k}\leq0$
for any $t\in[0,T]$.
Obviously, $z^l_R(x,t)\leq w^l_R(x,t),$
then we have
\begin{eqnarray*}
w^l&\geq& w^l_R+t\frac{-(\frac{B'}{B}-\frac{k'}{2k})w^l_R-(\frac{B'}{B}+\frac{k'}{2k})w^l_R
-BB'(w^l_R)^2(z^l_R)}{D(x,t,h)}\\
&=&w^l_R\left[1-t\frac{\frac{2B'}{B}+BB'w^l_Rz^l_R}{D(x,t,h)}\right],\\
z^l&\leq&z^l_R+t\frac{-(\frac{B'}{B}-\frac{k'}{2k})z^l_R-(\frac{B'}{B}+\frac{k'}{2k})z^l_R
-BB'(z^l_R)^2(w^l_R)}{D(x,t,h)}\\
&=&z^l_R\left[1-t\frac{\frac{2B'}{B}+BB'w^l_Rz^l_R}{D(x,t,h)}\right].
\end{eqnarray*}
From the priori assumption  of $w^l(x,t),  z^l(x,t),$
there exists a constant $h_2(P(T))>0$, such that,  when $0<h<h_2<h^*,$ for $(x,t)\in\Pi_T,$
$$\frac{h}{D(x, t, h)}\left|\frac{2B'}{B}+BB'w^l_Rz^l_R\right|\leq\frac{1}{2},$$
from which
$$1-t\frac{\frac{2B'}{B}+BB'w^l_Rz^l_R}{D(x,t,h)}
\geq 1-\frac{h}{D(x, t, h)}\left|\frac{2B'}{B}+BB'w^l_Rz^l_R\right|\geq\frac{1}{2}>0.$$
Hence $w^l(x, t)\geq0 ,z^l(x, t)\leq0 $  for  $t\in[0, h].$
Suppose that $w^l(x,t)\geq0$ and  $z^l(x,t)\leq0$ hold for $t<nh$. For $t\in[nh,(n+1)h)$,
we have
\begin{eqnarray}
w^l&\geq&w^l_R\left[1-(t-nh)\frac{\frac{2B'}{B}+BB'w^l_Rz^l_R}{D(x,t,h)}\right]\geq0,
\label{wl}\\
z^l&\leq&z^l_R\left[1-(t-nh)\frac{\frac{2B'}{B}+BB'w^l_Rz^l_R}{D(x,t,h)}\right]\leq0,
\label{zl}
\end{eqnarray}
when $0<h<h_2, $ since
$$1-(t-nh)\frac{\frac{2B'}{B}+BB'w^l_Rz^l_R}{D(x,t,h)}
\geq 1-\frac{h}{D(x, t, h)}\left|\frac{2B'}{B}+BB'w^l_Rz^l_R\right|\geq\frac{1}{2}>0.$$
Take $h_0=\min\{h_1, h_2\}<h^*$, and by induction, when $h<h_0$,   we conclude
that $w^l(x,t)\geq0$ and $z^l\leq0$ for $(x,t)\in\Pi_T.$ \par
Therefore, in $\Pi_T$, $w^l(x,t)\geq0, z^l(x,t)\leq0$, $-BB'(w^l_R)^2z^l_R\leq0.$
Then, for any $t\in[nh,(n+1)h)$,
\begin{equation*}
\begin{split}
w^l(x,t)&\leq w^l_R(x,t)+(t-nh)\frac{-(\frac{B'}{B}-\frac{k'}{2k})w^l_R(x,t)
-(\frac{B'}{B}+\frac{k'}{2k})z^l_R(x,t)}{D(x,t,h)},\\
z^l(x,t)
&\geq z^l_R(x,t)+(t-nh)\frac{-(\frac{B'}{B}-\frac{k'}{2k})z^l_R(x,t)
-(\frac{B'}{B}+\frac{k'}{2k})w^l_R(x,t)}{D(x,t,h)}\\
\end{split}
\end{equation*}
Since $D(x,t,h)$ tends to $1$ as $h$ tends to $0$, we can estimate
$$\frac{1}{D(x,t,h)}\in\left[\frac{2}{3}, 2\right]$$
under the a priori assumption.  In particular, from the construction of approximate solutions and
the property of Riemann solutions,
\begin{eqnarray*}
\sup_xw^l(x,(n+1)h-0)&\leq&\sup_xw^l_R(x,(n+1)h-0)\\
&&+2h\bigg[\left(\frac{k'}{2k}-\frac{B'}{B}\right)\sup_xw^l_R(x,(n+1)h-0)\\
&&-\left(\frac{B'}{B}+\frac{k'}{2k}\right)\inf_xz^l_R(x,(n+1)h-0)\bigg],\\
&\leq&\sup_xw^l(x,nh-0)\\
&&+2h\bigg[\left(\frac{k'}{2k}-\frac{B'}{B}\right)\sup_xw^l(x,nh-0)\\
&&-\left(\frac{B'}{B}+\frac{k'}{2k}\right)\inf_xz^l(x,nh-0)\bigg]
\end{eqnarray*}
and
\begin{eqnarray*}
\inf_xz^l(x,(n+1)h-0)&\geq&\inf_xz^l_R(x,(n+1)h-0)\\
&&+2h\bigg[\left(\frac{k'}{2k}-\frac{B'}{B}\right)\inf_xz^l_R(x,(n+1)h-0)\\
&&-\left(\frac{B'}{B}+\frac{k'}{2k}\right)\sup_xw^l_R(x,(n+1)h-0)\bigg]\\
&\geq&\inf_xz^l(x,nh-0)\\
&&+2h\bigg[\left(\frac{k'}{2k}-\frac{B'}{B}\right)\inf_xz^l(x,nh-0)\\
&&-\left(\frac{B'}{B}+\frac{k'}{2k}\right)\sup_xw^l(x,nh-0)\bigg].
\end{eqnarray*}
Let $P_n=\max\{\sup_xw^l(x,nh-0), -\inf_xz^l(x, nh-0)\}.$ Then we obtain
\begin{equation*}
\begin{split}
&\max\{\sup_xw^l(x,(n+1)h-0),-\inf_xz^l(x,(n+1)h-0)\}\\
&\leq P_n+2\left[\left(\frac{k'}{2k}-\frac{B'}{B}\right)+
\left(\frac{B'}{B}+\frac{k'}{2k}\right)\right]P_nh+O(h^2)\\
&=P_n+2\frac{k'}{k}P_nh.
\end{split}
\end{equation*}
Thus,
$$P_{n+1}\leq P_n+2\frac{k'}{k}P_nh.$$
Therefore, by induction, combining the ordinary differential equation
$$\frac{dP(t)}{dt}=2\frac{k(t)'}{k(t)}P(t),$$
whose solution is $P_0\frac{k(t)^2}{k(0)^2},$
one has
$$P_{n}\leq P(T)=P_0\frac{k(T)^2}{k(0)^2}.$$
Then for any $0<h<h_0,$ the  inequality (\ref{Rie}) follows.
\smallskip

\noindent \textit{Step 2:} We shall show the  lower bound of $w^l$,
 we see immediately
 from \eqref{wl}, along with the fact that $k(t), B(t)$ are given
$C^{1,1}$ functions in $[0,T]$, for $nh\leq t<(n+1)h$,
\begin{eqnarray*}
\inf_x w^l(x,t)&\geq&\inf_x w^l_R(x,t)\left[1-h\sup_x\frac{|\frac{2B'}{B}+BB'w^l_Rz^l_R|}{D(x,t,h)}\right]\\
&\geq&\inf_xw^l_R(x, t)(1-A(T)h),
\end{eqnarray*}
where $A(T)$ is defined as
$$A(T):=\sup_{(x, t)\in\Pi_T}\frac{|\frac{2B'}{B}+BB'w^l_Rz^l_R|}{D(x,t,h)}.$$
Hence
\begin{eqnarray*}
\inf_x w^l(x,(n+1)h-0)&\geq&\inf_x
w^l_R(x,(n+1)h-0)(1-A(T)h)\\
&\geq&\inf_xw^l(x,nh-0)(1-A(T)h.
\end{eqnarray*}
Let $W_n=\inf_x w^l(x, nh-0)),$ then
\begin{eqnarray*}
W_{n+1}&\geq&W_n(1-A(T)h)\geq w_0(1-A(T)h)^n\\
&\geq& w_0(1-A(T)h)^{T/h}\geq w_0e^{-A(T)T}.
\end{eqnarray*}
with $w_0=\inf_xw^l(x,0)\geq\delta_0.$
This completes the proof.
\end{proof}

\smallskip

\bigskip

\section{$H^{-1}_{loc}$ Compactness}

With  the uniform bound of the approximate solutions obtained in the previous Section,
we now prove the  $H^{-1}_{loc}$ compactness.

First, we recall the following embedding lemma in \cite{DCL}:
\begin{Lemma}\label{Th0}
If $\Omega\in\mathbb{R}^n$ is an open and bounded set, then
\begin{equation}
\begin{array}{ll}
(\text{compact set of } W^{-1,q}(\Omega))\cap(\text{compact set of } W^{-1,r}(\Omega))\\
\subset(\text{compact set of } W^{-1,2}(\Omega)),
\end{array}
\end{equation}
where $q$ and $r$ are  constants, $1<q\leq2<r.$
\end{Lemma}

Then we state our  following result on $H^{-1}$  compactness.
\begin{Theorem}\label{Th2}
Under the conditions of Theorem \ref{Th1}  and if the initial data
satisfies additionally the following condition:
$$\int_\mathbb{R}\frac{m_0^2+1}{2\rho_0}\leq C_0,$$
for some constant $C_0>0$, then
$\tilde{M}^l_x-\tilde{L}^l_t$ and $\tilde{N}^l_x-\tilde{M}^l_t$
are $H^{-1}_{loc}(\Omega)$ compact for  any open bounded set  $\Omega\subset\Pi_T$, where
$\tilde{L}^l=\rho^l$, $\tilde{M}^l=-m^l$, and $\tilde{N}^l=\frac{(m^l)^2-1}{\rho^l}$.
\end{Theorem}


\begin{proof}
For simplicity of notations, we drop the  index $l$ in the approximate solutions $U^l=(\rho^l, m^l)$.
For any entropy and entropy flux pair
$(\eta(U), q(U))$ of \eqref{Chap2} satisfying
$\nabla\eta\nabla f=\nabla q$
and any
$\varphi(x,t)\in C^\infty_0(\mathbb{R}\times[0,T]),$ we have
\begin{equation}\label{1}
\int_0^T\int_\mathbb{R}\eta(U)\varphi_t+q(U)\varphi_xdxdt=I_1(\varphi)+
I_2(\varphi)+I_3(\varphi)+I_4(\varphi),
\end{equation}
where
\begin{eqnarray*}
I_1(\varphi)&=&\int_\mathbb{R}\varphi(x,T)\eta(U_R(x,T))dx-
\int_\mathbb{R}\varphi(x,0)\eta(U_R(x,0))dx,\\
I_2(\varphi)&=&\int_0^T\int_\mathbb{R}[\eta(U)-\eta(U_R)]\varphi_t+
[q(U)-q(U_R)]\varphi_xdxdt,\\
I_3(\varphi)&=&\sum_{j,n}\int_{(j-1)l}^{(j+1)l}
[\eta(U_R)(x,nh-0)-\eta(U_R)(jl,nh)]\varphi(x,nh)dx\\
&=&I_{31}(\varphi)+I_{32}(\varphi)+I_{33}(\varphi),\\
I_{31}(\varphi)&=&\sum_{j,n}\int_{(j-1)l}^{(j+1)l}
[\eta(U_R)(x,nh-0)-\eta(U)(x,nh-0)]\varphi(x,nh)dx,\\
I_{32}(\varphi)&=&\sum_{j,n}\int_{(j-1)l}^{(j+1)l}
[\eta(U_R)(jl,nh)-\eta(U)(x,nh-0)]\varphi(jl,nh)dx,\\
I_{33}(\varphi)&=&\sum_{j,n}\int_{(j-1)l}^{(j+1)l}
[\eta(U_R)(x,nh-0)-\eta(U)(jl,nh)][\varphi(x,nh)-\varphi(jl,nh)]dx,\\
I_4(\varphi)&=&\int_0^T\sum\{\sigma[\eta]_0-[q]_0\}\varphi(x(t),t)dt,\\
\end{eqnarray*}
 $S=\{(x(t),t)\}$ denotes the contact discontinuity wave in $U_R(x,t)$,  and $[\eta]_0,$
$[q]_0$ denote the jump across $S$ of $\eta(U_R(x,t))$, $q(U_R(x,t))$, respectively.

\smallskip
\noindent\emph{Step1:}  
For the entropy and entropy flux pair $$\eta_*=\frac{m^2+1}{2\rho}, \quad q_*=\frac{-m^3+m}{2\rho^2},$$  which is convex
in $(\rho,m)$,  one has, from the assumption on $(\rho_0,m_0)(x)$,
\begin{equation*}
\int_\mathbb{R}\eta_*(\rho_0,m_0)dx=\int_\mathbb{R}\frac{m_0^2+1}{2\rho_0}\leq C_0.
\end{equation*}
Taking $\varphi\equiv1, \; \eta=\eta_*,\;  q=q_*$ in (\ref{1}), we have
 $ |I_1(1)|\leq C,  I_2(1)=0, $ and $|I_3(1)|+|I_4(1)|\leq C$.
 Here and in the rest of
the paper $C$ denotes a generic  constant depending on $T$,
but independent of $j,n,h,l$.
For $I_3$, one has
\begin{eqnarray*}
I_3(1)&=&\sum_{j,n}\int_{(j-1)l}^{(j+1)l}[\eta_*(U_R(x,nh-0))-\eta_*(U_R(jl,nh))]dx\\
&=&\sum_{j,n}\int_{(j-1)l}^{(j+1)l}[\eta_*(U_R(x,nh-0)-\eta_*(U(x,nh-0))]dx\\
&&+\sum_{j,n}\int_{(j-1)l}^{(j+1)l}[\eta_*(U(x,nh-0))-\eta_*(U(jl,nh))]dx\\
&=&\sum_{j,n}\int_{(j-1)l}^{(j+1)l}\int_0^1\nabla\eta_*(U_R(x,nh-0)+
\theta(U(x,nh-0)-U_R(x,nh-0)))d\theta \\
&&\quad \times (U_R(x,nh-0)-U(x,nh-0))dx\\
&&+\sum_{j,n}\int_{(j-1)l}^{(j+1)l}\int_0^1(1-\theta)\nabla^2\eta_*(U(jl,nh-0)+
\theta(U(x,nh-0)-U(jl,nh))]d\theta\\
&&\quad \times (U_R(jl,nh-0)-U(x,nh-0))^2dx,
\end{eqnarray*}
where we have used the mean-value  formulas:
\begin{eqnarray*}
g(a)-g(b)&=&\int_0^1g^\prime(a+\theta(b-a))d\theta(a-b),\\
g(a)-g(b)&=&g^\prime(b)(a-b)+(1-\theta)\int_0^1g^{\prime\prime}(b+\theta(a-b))d\theta(a-b)^2,
\end{eqnarray*}
 and the construction $U_R(jl,nh)=\frac1{2l}\int_{(j-1)l}^{(j+1)l}U(x,nh-0)dx.$
 From the $L^\infty$ estimate in Section 3, we have
 $$|U_R(x,nh-0)-U(x,nh-0))|=|H(U_R(x, t), x, t)(t-nh)|\leq Ch,$$
 and then
 \begin{eqnarray*}
 &&\bigg|\sum_{j,n}\int_{(j-1)l}^{(j+1)l}\int_0^1\nabla\eta_*(U_R(x,nh-0)+
\theta(U(x,nh-0)-U_R(x,nh-0)))d\theta \\
&&\quad\times (U_R(x,nh-0)-U(x,nh-0))dx\bigg|\\
&&\leq\sum_{j,n}\int_{(j-1)l}^{(j+1)l}\int_0^1|\nabla\eta_*(U_R(x,nh-0)+
\theta(U(x,nh-0)-U_R(x,nh-0)))|d\theta dx\cdot Ch\\
&&\leq C.
\end{eqnarray*}
Since $\eta_*$ is a convex entropy, the Hessen matrix is positive definite,
i.e., there exists a constant $c_0>0$ such that  $\nabla^2\eta_*\geq c_0I$.
In addition, for the convex entropy and entropy flux pair $(\eta_*, q_*)$, we have the following properties similar to those in  \cite{DCL}:
$\sigma[\eta_*]_0-[q_*]_0\geq0,$ and
\begin{equation*}
\int_0^T\sum\{\sigma[\eta_*]_0-[q_*]_0\}dt\leq C.
\end{equation*}
Thus,
\begin{eqnarray*}
&&\sum_{j,n}\int_{(j-1)l}^{(j-1)l}|U_R(jl,nh)-U(x,nh-0)|^2dx\\
&&\leq\frac{1}{c_0}\sum_{j,n}\int_{(j-1)l}^{(j+1)l}\int_0^1(1-\theta)\nabla^2\eta_*(U(jl,nh-0)+
\theta(U(x,nh-0)-U(jl,nh))d\theta\\
&&\quad \times |U_R(jl,nh-0)-U(x,nh-0)|^2dx\\
&&\leq C.
\end{eqnarray*}

\noindent
\emph{Step 2:} For any bounded set $\Omega\subset\mathbb{R}\times[0,T],$  any weak entropy and entropy flux pair
$(\eta, q)$, and any $\varphi\in C^\infty_0(\Omega),$ 
the uniform bound implies that,
\begin{eqnarray*}
|I_1(\varphi)|&\leq& C||\varphi||_{C_0(\Omega)},\\
|I_4(\varphi)|&\leq& C||\varphi||_{C_0(\Omega)}\int_0^T\sum\{\sigma[\eta_*]_0-[q_*]_0\}dt\\
&\leq& C||\varphi||_{C_0(\Omega)},\\
|I_{31}(\varphi)|&\leq& C\sum_{j,n}\int_{(j-1)l}^{(j+1)l}|H(U_R(x,nh-0),x,t)h||\varphi(x,nh)|dx\\
&\leq&C||\varphi||_{C_0(\Omega)},\\
|I_{32}(\varphi)|&\leq&C||\varphi||_{C_0(\Omega)}\sum_{j,n}
\int_{(j-1)l}^{(j+1)l}|U_R(x,nh-0)-U_R(jl,nh)|^2dx\\
&\leq&C||\varphi||_{C_0(\Omega)},
\end{eqnarray*}
which yields that
\begin{equation}\label{2}
I_1+I_4+I_{31}+I_{32} \text{ is compact in } W^{-1,\alpha_1},
\end{equation}
for some $1<\alpha_1<2$.
Moreover,
\begin{eqnarray*}
|I_{33}(\varphi)|&\leq&\sum_{j,n}\bigg[\int_{(j-1)l}^{(j+1)l}
(\eta(U_R(x,nh-0))-\eta(U_R(jl,nh)))^2dx\bigg]^{\frac{1}{2}}\\
&& \quad\times\bigg[\int_{(j-1)l}^{(j+1)l}
(\varphi(x,nh-0)-\varphi(jl,nh))^2dx\bigg]^{\frac{1}{2}}\\
&\leq&C\sum_{j,n}\bigg[\int_{(j-1)l}^{(j+1)l}
(U_R(x,nh-0)-U_R(jl,nh))^2dx\bigg]^{\frac{1}{2}}\\
&&\quad\times ||\varphi||_{C_0^{\alpha_2}(\Omega)}
\bigg[\sum_{j,n}\int_{(j-1)l}^{(j+1)l}|x-jl|^{2\alpha_2}dx\bigg]^{\frac{1}{2}}\\
&\leq&C||\varphi||_{C_0^{\alpha_2}(\Omega)}l^{\alpha_2-\frac{1}{2}}\\
&\leq&C||\varphi||_{W_0^{1,p}(\Omega)}l^{\alpha_2-\frac{1}{2}}
\end{eqnarray*}
for some  $\frac{1}{2}<\alpha_2<1$ and $1<p<\frac{2}{1+\alpha_2}$,
which implies that
\begin{equation}\label{3}
I_{33}  \text { is compact in } W^{-1,p}.
\end{equation}
Since
 $$|I_2(\varphi)\leq Ch\int_0^T\int_\mathbb{R}
 |\varphi_t|+|\varphi_x|dxdt\leq Ch||\varphi||_{H_0^1(\Omega)},$$
then
\begin{equation}\label{4}
I_2  \text{ is compact in } H^{-1}_0(\Omega).
\end{equation}
From (\ref{2})-(\ref{4}),
\begin{equation}\label{5}
I_1+I_2+I_3+I_4 \text{ is compact in } W_0^{-1,r},
\end{equation}
where $r=\min\{p, 2, \alpha_1\}.$
In addition,
\begin{eqnarray*}
\int_0^T\int_\mathbb{R}\eta(U)\varphi_t+q(U)\varphi_xdxdt
\leq C||\varphi||_{W_0^{1,1}(\Omega)},
\end{eqnarray*}
then,
\begin{equation}\label{6}
I_1+I_2+I_3+I_4 \text{ is compact in } W^{-1,\infty}.
\end{equation}
Combining  (\ref{5}) and (\ref{6}), from the embedding Lemma
\ref{Th0}, we have
\begin{equation}\label{7}
\eta_t+q_x=I_1+I_2+I_3+I_4 \text{ is compact in } H^{-1}_{loc}(\Omega).
\end{equation}
Particularly, we choose
$(\eta, q)=(\rho, m)$ and $(\eta, q)=(m, \frac{m^2-1}{\rho})$
to conclude that
$$\tilde{M}_x-\tilde{L}_t, \; \tilde{N}_x-\tilde{M}_t
\text{ are } H^{-1}_{loc}(\Omega) \text{ compact}.$$
\end{proof}

\section{Main Theorem and Its Proof}

In this Section, we first state our main theorem and then provide its proof.

\begin{Theorem}[Main Theorem]\label{Th3}
Suppose that  $(M,g)$ is a 2-dimensional manifold with smooth metric $g=B(y)^2dx^2+dy^2$,
$(x, y)\in\mathbb{R}\times[-T,0]$, for any given $T>0,$ the  Gauss curvature is  negative
given by $K(y)=-k(y)^2$, where $k(y)$ and $B(y)$ are positive functions satisfying
\eqref{B}, and $\ln (B^2k)$ is a nondecreasing $C^{1,1}$ function.
 If the initial data satisfies the following conditions:
\begin{equation}\label{ini1}
\delta_0\leq w(x, 0)\leq P_0,~-P_0\leq z(x, 0)\leq0\,  (\text{or }
0\leq w(x, 0)\leq P_0,~-P_0\leq z(x, 0)\leq-\delta_0)
\end{equation}
for some constants $\delta_0>0, P_0>0$ and all $x\in\mathbb{R}$.
and
\begin{equation}\label{ini2}
\int_\mathbb{R}\frac{m(x,0)^2+1}{\rho(x,0)}\leq C_0,
\end{equation}
for some positive constants $\delta_0$, $P_0$ and $C_0$,
then \eqref{Chap}with the above given initial data
has a $L^\infty $ weak solution, thus the Gauss-Codazzi equations
has a $L^\infty $ weak solution.
\end{Theorem}

\begin{proof}
As in  Section 2, let $y=t-T$, then $t\in[0,T].$ With the
results obtained in Sections 3 and 4, we can use the framework of \cite{CSW} to prove the above main theorem as follows.
\smallskip

\noindent
\emph{Step 1:} From Theorem \ref{Th1}, we have
\begin{equation*}
\begin{split}
&\delta_0e^{-A(T)T}\leq w^l(x, t)\leq P(T),~-P(T)\leq z^l(x, t)\leq0\\
&(\text{or }
0\leq w^l(x, t)\leq P(T),~-P(T)\leq z^l(x, t)\leq-\delta_0e^{-A(T)T}).
\end{split}
\end{equation*}
Note that $\rho^l=\frac{2}{w^l-z^l}$ and $m^l=\frac{w^l+z^l}{w^l-z^l},$ we get
$$\frac{1}{P(T)}\leq\rho^l(x, t)\leq\frac{2e^{A(T)T}}{\delta_0},~
|m^l(x,t)|\leq1.$$
Recalling the relations of the approximate solutions:
\begin{eqnarray*}
&&L^l(x,t)=k(t)\tilde{L}^l(x,t)=k(t)\rho^l(x,t),\\
&&M^l(x,t)=k(t)\tilde{M}^l(x,t)=-m^l(x,t)k(t),\\
&&N^l(x,t)=k(t)\tilde{N}^l(x,t)=\frac{(m^l(x,t))^2-1}{\rho^l(x,t)}k(t),
\end{eqnarray*}
one can also has the $L^\infty$ uniform boundedness of $L^l, M^l, N^l,$ since
$k(t)$ is a given $C^{1, 1}$ function.\par
By Theorem \ref{Th2} combining with the initial conditions \eqref{ini1},
\eqref{ini2} and the relation
$(L^l, M^l, N^l)=k(t)(\tilde{L}^l, \tilde{M}^l,\tilde{ N}^l)$, together
with $k(t)\in C^{1, 1},$
we can easily obtain that the approximate solutions $(L^l, M^l, N^l)$ are also
$H^{-1}_{loc}$ compact, and we omit the details.

\smallskip

\noindent
\emph{Step 2:} 
For any $\varphi=\varphi(x,t)\in C_0^\infty(\Pi_T),$
\begin{eqnarray*}
&&\iint_{\Pi_T}(U^l(x,t)\varphi_t+f(U^l(x,t))\varphi_x+H(U^l(x,t),x,t)\varphi)dxdt
+\int_\mathbb{R}\varphi(x,0)U^l(x,0)dx\\
&=&\iint_{\Pi_T}(U^l_R(x,t)\varphi_t+f(U^l_R(x,t))\varphi_x+H(U^l_R(x,t),x,t)\varphi)dxdt
+\int_\mathbb{R}\varphi(x,0)U^l(x,0)dx\\
&&
+\sum_{n}\int_{nh}^{(n+1)h}\int_\mathbb{R}\bigg[(U^l(x,t)-U^l_R(x,t))\varphi_t+
(f(U^l(x,t))-f(U^l_R(x,t))\varphi_x\\
&&\qquad\qquad\qquad\qquad  +(H(U^l(x,t),x,t)-H(U^l_R(x,t),x,t))\varphi\bigg]dxdt\\
&=&J_1+J_2.
\end{eqnarray*}
Since $U^l, \,U^l_R$ are uniformly bounded,
then
$$|U^l(x,t)-U^l_R(x,t)|=|H(U^l_R(x,t),x,t)(t-nh)|\leq Ch,$$
and
\begin{eqnarray*}
&&|H(U^l(x,t),x,t)-H(U^l_R(x,t),x,t)|\\
&\leq&C\left|\frac{(m^l(x,t))^2-1}{\rho^l(x,t)}-\frac{(m^l_R(x,t))^2-1}{\rho^l_R(x,t)}\right|\\
&& \quad+C|U^l(x,t)-U^l_R(x,t)|\\
&\leq&Ch+C\bigg[\frac{|(m_\theta(x,t))^2-1|}{(\rho_\theta(x,t))^2}|\rho^l(x,t)-\rho^l_R(x,t)|\\
&&\qquad\qquad +\frac{2|m_\theta|}{\rho_\theta}|m^l(x,t)-m^l_R(x,t)|\bigg]\\
&\leq& Ch,
\end{eqnarray*}
where $\rho_\theta$, $m_\theta$ are some functions with values between
$\rho^l$ and $\rho^l_R$, $m^l$ and $m^l_R$ respectively, satisfying
$$\frac{2|m_\theta|}{\rho_\theta}\leq C, \qquad\frac{|(m_\theta(x,t))^2-1|}{(\rho_\theta(x,t))^2}\leq C.$$
Since
\begin{eqnarray*}
|f(U^l(x,t))-f(U^l_R(x,t))|&\leq&\int_0^1|
f^\prime(U^l_R(x,t)+\theta(U^l(x,t)-U^l_R(x,t)))|
d\theta\\
&&\times |U^l(x,t)-U^l_R(x,t)|\\
&\leq& Ch,
\end{eqnarray*}
we then have
\begin{equation*}
|J_2|\leq Ch\iint_{\Pi_T}(|\varphi_t|+|\varphi_x|+|\varphi|)dxdt\leq Ch.
\end{equation*}
Rewrite $J_1$ as
\begin{eqnarray*}
J_1&=&\sum_{n}\int_\mathbb{R}\varphi(x,nh)(U^l_R(x,nh-0)-U^l_R(jl,nh))dx\\
&&\quad+\int_0^T\int_\mathbb{R}\varphi(x,t)H(U^l_R(x,t), x, t)dxdt\\
&=&\sum_{j,n}\int_{(j-1)l}^{(j+1)l}(\varphi(x,nh)-\varphi(jl,nh))(U^l_R(x,nh-0)-U^l_R(jl,nh))dx\\
&&+\sum_{j,n}\varphi(jl,nh)\int_{(j-1)l}^{(j+1)l}(U^l_R(x,nh-0)-U^l_R(jl,nh))dxdt\\
&&\quad+\int_0^T\int_\mathbb{R}\varphi(x,t)H(U^l_R(x,t), x, t)dxdt\\
&=&J_{11}+J_{12},
\end{eqnarray*}
where
\begin{eqnarray*}
|J_{11}|&\leq&||\varphi||_{C^1_0(\Pi_T)}\bigg(\sum_{j,n}\int_{(j-1)l}^{(j+1)l}|x-jl|^2dx\bigg)^{\frac{1}{2}}\\
&&\times\bigg(\sum_{j,n}\int_{(j-1)l}^{(j+1)l}|U^l_R(x,nh-0)-U^l_R(jl,nh)|^2dx\bigg)^{\frac{1}{2}}\\
&\leq&Cl^{\frac{1}{2}}||\varphi||_{C^1_0(\Pi_T)}\leq Cl^{\frac{1}{2}}
\end{eqnarray*}
from the a priori estimate in Section 4.
Inserting
$$U^l_R(jl,nh)=\frac{1}{2l}\int_{(j-1)l}^{(j+1)l}U^l(x,nh-0)dx$$ into $J_{12}$, we obtain
\begin{equation*}
\begin{split}
|J_{12}|&=\bigg|\int_0^T\int_\mathbb{R}\varphi(x,t)H(U^l_R(x,nh-0), x, t)dxdt
+\sum_{j,n}\varphi(jl,nh)\int_{(j-1)l}^{(j+1)l}\bigg[U^l_R(x,nh-0)\\
&\qquad-\frac{1}{2l}\int_{(j-1)l}^{(j+1)l}(U^l(x,nh-0)dx+H(U^l_R(x,t),x,nh-0))h\bigg]dx\bigg|\\
&=\bigg|-h\sum_{j,n}\varphi(jl,nh)\int_{(j-1)l}^{(j+1)l}
H(U^l_R(x,nh-0),x,nh-0)dx\\
&\qquad+\int_0^T\int_\mathbb{R}\varphi(x,t)H(U^l_R(x,nh-0), x, t)dxdt\bigg|\\
&\leq\sum_{j,n}\int_{(n-1)h}^{nh}\int_{(j-1)l}^{(j+1)l}\bigg|
(\varphi(x,t)-\varphi(jl,nh))H(U^l_R(x,t),x,t)\bigg|dxdt\\
&\quad +\sum_{j,n}\int_{(n-1)h}^{nh}\int_{(j-1)l}^{(j+1)l}\bigg|
\varphi(jl,nh)(H(U^l_R(x,t),x,t)-H(U^l_R(x,nh-0),x,nh-0))\bigg|dxdt\\
&=J_{121}+J_{122}.
\end{split}
\end{equation*}
For $J_{121},$ since $H(U^l_R(x,t),x,t)$ is bounded, we have
\begin{eqnarray*}
J_{121}&=&\sum_{j,n}\int_{(n-1)h}^{nh}\int_{(j-1)l}^{(j+1)l}
|\varphi(x,t)-\varphi(jl,nh)||H(U^l_R(x,t),x,t)|dxdt\\
&\leq&C\sum_{j,n}\int_{(n-1)h}^{nh}\int_{(j-1)l}^{(j+1)l}
||\varphi||_{C^1_0(\Pi_T)}(|x-jl|+|t-nh|)dxdt\\
&\leq&Cl||\varphi||_{C^1_0(\Pi_T)}\\
&\leq&Cl.
\end{eqnarray*}
For $J_{122},$  from the expression of $H(U,x,t)$, regularity of $\Gamma^i_{jk},$
and the $C^1$ continuity of $H(U,x,t)$ in its variables,  we have,
\begin{equation*}
\begin{split}
J_{122}&\leq C\sum_{(jl, nh)\in\text{ supp }\varphi}\int_{(n-1)h}^{nh}\int_{(j-1)l}^{(j+1)l}
|U^l_R(x,t)-U^l_R(x,nh-0)|+|t-nh|dxdt\\
&\leq C\left(\sum_{(jl, nh)\in\text{ supp }\varphi}\int_{(n-1)h}^{nh}dt\int_{(j-1)l}^{(j+1)l}
|U^l_R(x,t)-U^l_R(x,nh-0)|^2dx\right)^{\frac{1}{2}}+Cl.
\end{split}
\end{equation*}
To estimate
$$\int_{(j-1)l}^{(j+1)l} |U^l_R(x,t)-U^l_R(x,nh-0)|^2dxdt$$
we consider $U^l_R(x,t)$ which is the Riemann solution of (\ref{HO}) with the initial data:
\begin{eqnarray*}
U^l_R(x, (n-1)h)=\left\{\begin{array}{ll}
U_1, x<jl,\\
U_2, x>jl,
\end{array}\right.
\end{eqnarray*}
in the rectangle $\{(x,t): (n-1)h\leq t<nh, (j-1)l<x<(j+1)l\}.$
Due to the linear degeneracy,  $U^l_R$ is the  contact
discontinuity wave defined by the following:
\begin{eqnarray*}
U^l_R(x, t)=\left\{\begin{array}{ll}
U_1, \; x<jl+\lambda_k(U_1)(t-(n-1)h),\\
U_k(\epsilon), \; x>jl+\lambda_k(U_1)(t-(n-1)h),
\end{array}\right.
\end{eqnarray*}
where $k=1, 2$,  and $U_k(\epsilon)$ is the contact discontinuity curve satisfying
\begin{eqnarray*}
\left\{\begin{array}{ll}
\frac{U_k(\epsilon)}{d\epsilon}=r_k(U_k(\epsilon)),\\
U_k(0)=U_1.
\end{array}\right.
\end{eqnarray*}

Therefore, from Figure 2 below and definition of contact discontinuity wave,  $U^l_R(x, nh-0)$ is  a function of three pieces of constants
in the rectangle. From Lemma  4 in \cite{DCL}, one has
\begin{equation*}
\int_{(j-1)l}^{(j+1)l}|U^l_R(x, nh-0)-U^l_R(jl, nh)|^2dx\geq C d_0^3l\sum|\varepsilon(U^l_R(x, nh-0))|^2,
\end{equation*}
where $\sum|\varepsilon(U^l_R(x, nh-0))|$ denotes all of the jump strength of $U^l_R(x, nh-0)$ over  $((j-1)l, (j+1)l),$ and
$d_0$ denotes the minimum of ratios of lengths of constant state intervals of $2l$ for  $U^l_R(x, nh-0).$
\begin{figure}[th]
\begin{center}
  \includegraphics[width=12cm,height=8cm]{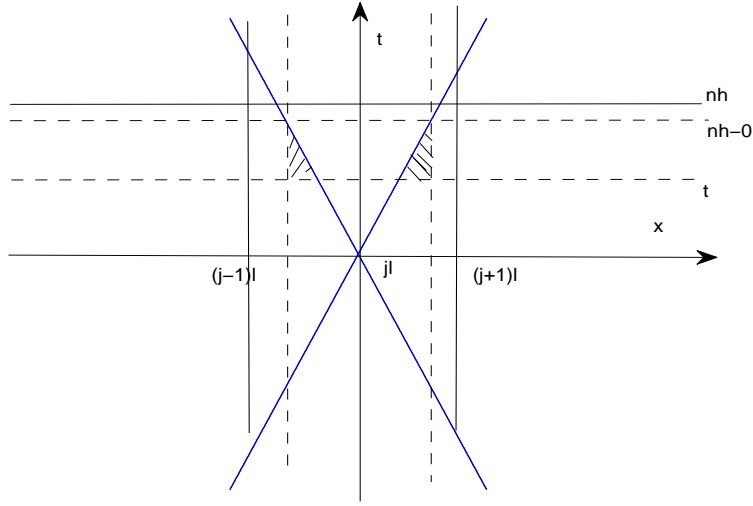}
  \caption{Graph of integral region }\label{figure}
  \end{center}
\end{figure}

Hence
\begin{equation*}
\begin{split}
d_0&=\frac{1}{2l}\min\{l+\lambda_1h, (\lambda_2-\lambda_1)h, l-\lambda_2h\}\\
&=\min\left\{\frac{1}{2}+z^l\frac{h}{2l} , \frac{2}{\rho^l}\frac{h}{2l},  \frac{1}{2}-w^l\frac{h}{2l}\right\}\\
&\geq\min\left\{\frac{1}{2}-P(T)\frac{l}{2h},  \delta_0e^{-A(T)T}\frac{h}{2l}\right\}:=\delta,
\end{split}
\end{equation*}
and then
\begin{equation*}
\begin{split}
&\int_{(j-1)l}^{(j+1)l}|U^l_R(x, t)-U^l_R(jl, nh-0)|^2dx\leq Cl\sum|\varepsilon(U^l_R(x, nh-0))|^2\\
&~~\leq C\delta^{-3} \int_{(j-1)l}^{(j+1)l}|U^l_R(x, nh)-U^l_R(jl, nh)|^2dx.
\end{split}
\end{equation*}
Thus,
 \begin{equation*}
 \begin{split}
  J_{122}&\leq C\left(\sum_{(jl, nh)\in\text{ supp }\varphi}\int_{(n-1)h}^{nh}dt
  C\delta^{-3}  \int_{(j-1)l}^{(j+1)l}|U^l_R(x, nh-0)-U^l_R(jl, nh)|^2dx\right)^{\frac{1}{2}}\\
 & \leq Cl^{\frac{1}{2}}+Cl.
 \end{split}
 \end{equation*}

From  all the above estimates, we have
$$J_1+J_2=|J_{11}|+J_{121}+J_{122}+|J_{2}|\leq Ch+Cl^{\frac{1}{2}}\le Cl+Cl^\frac12.$$
Then
\begin{eqnarray*}
&&\iint_{\Pi_T}U^l(x,t)\varphi(x,t)_t+f(U^l(x,t))\varphi(x,t)_x\\
&&\qquad+H(U^l(x,t),x,t)\varphi(x,t)dxdt+\int_\mathbb{R}U^l(x,0)\varphi(x,0)dx\\
&&\rightarrow0 \text{ as }l\rightarrow 0.
\end{eqnarray*}
That is,  there exists $o^l(1)=(o_1^l(1),o_2^l(1))^\top$ with $o^l(1)\rightarrow0$
 in the sense of distributions as $l\rightarrow0$ such that
$$U(x,t)_t+f(U(x,t))_x=H(U(x,t),x,t)+o^l(1),$$
which leads to the following
\begin{eqnarray*}
&& M^l_x-L^l_y=\Gamma^2_{22}L^l-\Gamma^2_{12}M^l+\Gamma^2_{11}N^l+o_1^l(1),\\
&& N^l_x-M^l_y=-\Gamma^1_{22}L^l+\Gamma^1_{12}M^l-\Gamma^1_{11}N^l+o_2^l(1),\\
&& L^lN^l-(M^l)^2=K,
\end{eqnarray*}
since $k(t)$ is given in $C^{1, 1}([0, T]).$

\smallskip

\noindent
\emph{Step 3:}  From the above two steps,  we conclude that the approximate solutions constructed by
the  Lax-Friedrichs scheme satisfy all the conditions of
the following compensated compactness framework in \cite{CSW}:
\begin{Lemma}\label{lemma2}
Let a sequence of functions $(L^\varepsilon, M^\varepsilon, N^\varepsilon)(x,y)$,
defined on an open subset $\Omega\subset\mathbb{R}^2$, satisfy the
following framework:
\begin{itemize}
\item[(W.1)]$(L^\varepsilon, M^\varepsilon, N^\varepsilon)(x,y)$ is uniformly
bounded almost everywhere in $\Omega\subset\mathbb{R}^2$ with respect to $\varepsilon$;
\item[(W.2)]$M^\varepsilon_x-L^\varepsilon_y$ and $N^\varepsilon_x-M^\varepsilon_y$ are
compact in $H^{-1}_{loc}(\Omega)$;
\item[(W.3)]There exist $o_j^\varepsilon(1)$, $j=1,2,3$, with $o_j^\varepsilon(1)\rightarrow0$
in the sense of distributions as $\varepsilon\rightarrow0$ such that
\begin{equation}\label{c}
\begin{split}
 &M_x^\varepsilon-L_y^\varepsilon=\Gamma^2_{22}L^\varepsilon-2\Gamma^2_{12}M^\varepsilon
 +\Gamma^2_{11}N^\varepsilon+o_1^\varepsilon(1),\\
 &N_x^\varepsilon-M_y^\varepsilon=-\Gamma^1_{22}L^\varepsilon+2\Gamma^1_{12}M^\varepsilon
 -\Gamma^1_{11}N^\varepsilon+o_2^\varepsilon(1),\\
 \end{split}
\end{equation}
and
\begin{equation}\label{G}
L^\varepsilon N^\varepsilon-(M^\varepsilon)^2=K+o_3^\varepsilon(1).
\end{equation}
\end{itemize}
Then there exists a subsequence (still labeled) $(L^\varepsilon, M^\varepsilon, N^\varepsilon)$
converging weak-star in $L^\infty$ to $(L, M, N)(x,y)$ as $\varepsilon\rightarrow0$ such that
\begin{itemize}
\item[(1)]$ (L, M, N)$ is also bounded in $\Omega\subset\mathbb{R}^2$;
\item[(2)]the Gauss equation (\ref{G}) is weakly continuous with respect to the
 subsequence $(L^\varepsilon, M^\varepsilon, N^\varepsilon)$ converging weak-star in
 $L^\infty$ to $(L, M, N)(x,y)$ as $\varepsilon\rightarrow0$;
\item[(3)] The Codazzi equations (\ref{c})  as $\varepsilon\to 0$ hold for $(L, M, N)$ in the sense
of distribution.
\end{itemize}
\end{Lemma}
Therefore, Lemma \ref{lemma2} implies that the Gauss-Codazzi system
has a $L^\infty$ weak solution in the sense of
distributions,  which yields a $C^{1,1}$ isometric immersion into $\mathbb{R}^3.$
The proof of the main theorem is complete. Besides, \eqref{Chap}
with the above given initial data also has a $L^\infty$ weak solution.
\end{proof}

\bigskip

\section*{Acknowledgments}
F. Huang's research was supported in part  by NSFC Grant No. 11371349,
National Basic Research Program of China (973 Program) under Grant
No. 2011CB808002, and the CAS Program for Cross $\&$ Cooperative
Team of the Science $\&$ Technology Innovation.
D. Wang's research was supported in part by the NSF Grant DMS-1312800 and NSFC Grant No. 11328102.
The authors would like to thank the referees for careful reading of the manuscript and for valuable comments.

\end{document}